\documentclass[a4paper,12pt]{article}
\usepackage{style}
\usepackage[english]{babel}

\title{A new integral equation for Brownian stopping problems with finite time horizon}

\date{\today}

\author{S\"oren Christensen\thanks{Mathematisches Seminar, Christian-Albrechts-Universit\"at zu Kiel, Heinrich-Hecht-Platz 6, 24118 Kiel, Germany, E-mail:  christensen@math.uni-kiel.de}, Simon Fischer\thanks{Mathematisches Seminar, Christian-Albrechts-Universit\"at zu Kiel, Heinrich-Hecht-Platz 6, 24118 Kiel, Germany, E-mail:  fischer@math.uni-kiel.de}}

\begin{document}

\maketitle

\begin{abstract}
For classical finite time horizon stopping problems driven by a Brownian motion
\begin{equation*}
    V(t,x) = \sup_{t\leq\tau\leq0}\e_{(t,x)}[g(\tau,W_{\tau})],
\end{equation*}
we derive a new class of Fredholm type integral equations for the stopping set. 
For a large class of discounted problems, we show by analytical arguments that the equation uniquely characterizes the stopping boundary of the problem.
Regardless of uniqueness, we use the representation to rigorously find the limit behavior of the stopping boundary close to the terminal time. Interestingly, it turns out that the leading-order coefficient is universal for wide classes of problems. We also discuss how the representation can be used for numerical purposes.
\end{abstract}


\section{Introduction}
Let $W$ be an $n$-dimensional standard Brownian motion started at time $t\leq 0$ in $\textbf{x} \in \rr ^{n}$ and $g:\rr_{\leq 0}\times \rr^{n} \to \rr$ a payoff function with properties to be specified later. We consider the stopping problem with finite time horizon
\begin{equation}\label{eq:sp1}
    V(t,\textbf{x}) = \sup_{t\leq\tau\leq0}\e_{(t,x)}[g(\tau,W_{\tau})].
\end{equation}
Note that we use 0 as the terminal time and we start the process from $t\leq 0$. The usefulness of this convention will become clear later.
Such problems arise in a wide variety of fields, including sequential statistics, change point detection, one-armed bandit problems starting in the 1960s, and option pricing and economics in the last decades. We refer to \cite{lai05} for an overview (with a focus on the contributions by H. Chernoff). As no closed form solutions can be expected for most problems of interest, different approaches have been suggested to gain information about the solution. In the last decades, the most common analytical  approach is to characterize the (unknown) stopping boundary of the problem in terms of a nonlinear integral equation of  Volterra-type. The main mathematical difficulty here is to prove rigorously that the nonlinear integral equation  has a unique solution. We discuss this in more detail in Section \ref{sec:2}. For some more details and historical background we refer to \cite{peskir06} and \cite{peskir05}.

The contribution of this paper is to introduce and study a new integral equation for the boundary, which is -- in contrast to the standard equations for finite time horizon problems -- of Fredholm-type. 
For simplicity, we often restrict ourselves to the case with discounting in the following, i.e., we let $r\geq 0$ and consider payoff functions of the form $g(t,\textbf{x}) = e^{-rt}h(\textbf{x})$. Let further $\g = \frac{\partial}{\partial t}+\frac{1}{2}\Delta$ be the characteristic space-time operator of $(t,W_t)$, i.e.,
\begin{align}\label{eq:generator}
	-\g g(t,\textbf{x}) = e^{-rt}\big(r h(\textbf{x}) - \frac{1}{2}\Delta h(\textbf{x})\big) =:e^{-rt}\tilde h(\textbf{x})
	\end{align}
and $C$ the continuation set of \eqref{eq:sp1}. Our main observation is that -- under natural assumptions -- the following integral equation holds:
\begin{equation}\label{eq:main5}
      0 = \int_{C}e^{\textbf{c}\cdot \textbf{y} + \left(\frac{\norm{\textbf{c}}^2}{2}-r\right)s}\tilde h (y)\d (s,\textbf{y})
\end{equation}
 for $\textbf{c}\in \rr^{n}$ to be specified later. This is shown in Theorem \ref{th:martin} in Section \ref{sec:2}. In this section we also discuss a special one-dimensional case and some generalizations  
Questions that arise are firstly whether the representation determines $C$ uniquely and secondly whether it can be used to derive analytical properties of $C$. 

The second question is tackled in Section \ref{sec:3} where we analyze the limit behavior of the continuation set $C$ for $t\to 0$. This is done independently of uniqueness, showing the properties for every set that satisfies \eqref{eq:main5}. More precisely, we derive a second order approximation for $C$ close to $t=0$ in the one-dimensional case in Subsection \ref{ssec:1}. 
In Subsection \ref{ssec:n} we discuss how the described method can be extended to multidimensional problems.  

The question of uniqueness is discussed in Section \ref{sec:u}. We prove that a version of \eqref{eq:main5} determines $C$ uniquely in the one-dimensional and one-sided case, see Theorem \ref{th:eind}. For the proof we use methods known from the identifiability of certain mixtures of Gaussian laws which we take over to non-compact parameter spaces. The methods and results presented may therefore also be of interest from a statistical point of view. The proof is analytical in nature, in contrast to the uniqueness results for the standard integral equations which usually are based on probabilistic arguments. 

In Section \ref{sec:5} we describe numerical procedures based on \eqref{eq:main5} and give some examples.
We conclude with a brief discussion of generalizations and further applications in Section \ref{sec:con}.

\section{Fredholm representation}\label{sec:2}
In this section we derive a Fredholm type integral representation for a large class of stopping problems with an $n$-dimensional Brownian motion as a driving process.

\subsection{General discounted problems}
Let $W$ be an $n$-dimensional standard Brownian motion with space-time generator $\g=\frac{\partial}{\partial t}+\frac{1}{2}\Delta$ and transition kernel $p$. For the sake of simplicity, we assume $h:\rr^{n} \to \rr$ to be in $C^{2}$ and 
 want to analyze the stopping problem 
\begin{equation}\label{eq:sp12}
    V(t,\textbf{x}) = \sup_{t\leq\tau\leq0}\e_{(t,\textbf{x})}[e^{-r\tau}h(W_{\tau})],
\end{equation}
where the supremum is taken over all stopping times $\tau$ with $t\leq\tau\leq0$ and $W_t = \textbf{x}$ a.s. Note that, in contrast to most other references, our terminal time is denoted by 0 and we start the process from $t\leq 0$. We also use the time point 0 as the reference for discounting, i.e., we use $e^{-r\tau}$ instead of $e^{-r(\tau-t)}$. The usefulness of this conventions will become clear later.

\begin{remark}\label{rem:drift}
	If a Brownian motion with drift is the driving process $X$ of the stopping problem 
	\begin{equation}\label{eq:drift}
		V(t,\textbf{x}) = \sup_{t\leq\tau\leq0}\e_{(t,\textbf{x})}[e^{-r\tau}h(X_{\tau})],
	\end{equation}
	then we can convert \eqref{eq:drift} to our setting via a measure transformation and obtain a problem of the form
	\begin{equation}
		V'(t,\textbf{x}) = \sup_{t\leq\tau\leq0}\e_{(t,\textbf{x})}[e^{-r'\tau}h'(W_{\tau})],
	\end{equation} 
	where $h'$ is a transformed payoff function and $W$ is a standard Brownian motion. We refer to \cite{christensen2020} and \cite{lerche07} for details.
\end{remark}

We denote the continuation set by 
    \[C:=\{(t,\textbf{x})\in \rr_{\leq0}\times \rr^{n}\mid V(t,\textbf{x})>e^{-rt}h(\textbf{x})\}\]
and the stopping set by $S=C^c$. For a fixed time $t$ we set $C_t := \set{\textbf{x}\in\rr^{n}\mid (t,\textbf{x})\in C}$ and define $S_t$ accordingly. We denote the first entrance time to $S$ by
$\tau^{\ast} = \tau_{S} := \inf\{t\leq s\leq 0\mid W_{s} \in S \}$ and note that this stopping time is optimal under minimal assumptions by general theory. Under suitable assumptions (see the discussion in \cite{peskir06} based on \cite{MR2408999}, see also \cite{MR4133366} and \cite{cai2021change} for some generalizations), we can apply a generalized version of Dynkin's formula to \eqref{eq:sp12} and obtain -- using $\tilde h$ according to \eqref{eq:generator} and that $\g V(t,x) = e^{-rt}\tilde h(\textbf{x}) = 0$ on $C$ --
\begin{align*}
    V(t,\textbf{x}) 
    & = \e_{(t,\textbf{x})}[h(W_0)] + \e_{(t,\textbf{x})}\left[\int_{t}^{0}\ind_{\{W_s\in S\}} e^{-rs}\tilde h(W_s)\d s\right]\\
    & =\int_{-\infty}^{\infty}h(\textbf{y})p((t,\textbf{x}),(0,\textbf{y}))\d \textbf{y} + \int_{t}^{0}\int_{S_s}e^{-rs}\tilde h(\textbf{y})p((t,\textbf{x}),(s,\textbf{y})) \d \textbf{y} \d s,
\end{align*}
where $p$ denotes the Brownian transition kernel.
In the financial context, this representation is called early-exercise-premium decomposition. In the same way we get 
\begin{align*}
    e^{-rt}h(\textbf{x})
    & = \e_{(t,\textbf{x})}[h(W_0)] + \e_{(t,\textbf{x})}\left[\int_{t}^{0} e^{-rs}\tilde h(W_s)\d s\right]\\
    & =\int_{-\infty}^{\infty}h(\textbf{y})p((t,\textbf{x}),(0,\textbf{y}))\d \textbf{y} + \int_{t}^{0}\int_{\rr^{n}}e^{-rs}\tilde h(\textbf{y})p((t,\textbf{x}),(s,\textbf{y})) \d \textbf{y} \d s.
\end{align*}
For $(t,\textbf{x})\in S$, we have $V(t,\textbf{x}) = e^{-rt}h(\textbf{x})$ so if we subtract the two equations above we obtain
\begin{equation}\label{eq:pes}
    0 = \int_{t}^{0}\int_{C_s}e^{-rt}\tilde h(\textbf{y})p((t,\textbf{x}),(s,\textbf{y})) \d \textbf{y} \d s,~ \forall (t,\textbf{x})\in S.
\end{equation}
This is nowadays a standard representation for stopping problems which we use as a starting point for our approach, so that we silently assume it to hold in the following. 

Note that $h$ does not necessarily have to be in $C^{2}$ for this to work, but could also be a weak solution to \eqref{eq:generator} in some sense, see assumptions in the literature above. For notational convenience we, however, write $\tilde h$ as a function.

Motivated by heat equations in physics, a slightly more complicated version of this integral equation was introduced by van Moerbeke in  \cite{moerbeke76} for one-sided one-dimensional problems. Van Moerbeke also showed local uniqueness of the solution of the integral representation. In the early 90s Kim \cite{kim90} and Myneni \cite{myneni92} derived \eqref{eq:pes} for the optimal exercise boundary of American options. Evaluating \eqref{eq:pes} for $(t,\textbf{x})$ in the boundary of  the stopping set leads to a useful equation for describing the stopping boundary. It was not until 2005, however, that it was rigorously proven by Peskir that this equation uniquely determined the stopping boundary for the case of the American put in the Black-Scholes market, see \cite{peskir05}. The proof is based on probabilistic arguments and has later been adapted for many other classes of problems. 
The representation \eqref{eq:pes} is for example used to numerically approximate the continuation set of stopping problems. The numeric evaluation is however not always easy. One reason is that the integrand has singularities in $(t,\textbf{x})\in \partial C$. 

As mentioned above, the key result of \cite{peskir05} was that it is enough to evaluate \eqref{eq:pes} for $(t,\textbf{x})\in \partial S\subseteq S$. The idea of the approach we suggest here is to use another subset of $S$. More precisely, we compactify the stopping region and use the \textit{infinitely far away} boundary (the Martin boundary, see Subsection \ref{ssec:gen} below) instead.
This leads to our Fredholm representation from \eqref{eq:pes}. 
For simplicity we assume that the infinite time horizon problem 
\begin{equation}\label{eq:infiniteth}
    V(\textbf{x}) = \sup_{0\leq\tau}\e_{\textbf{x}}[e^{-r\tau}h(W_\tau)]
\end{equation}
is solvable, see Section \ref{ssec:gen} for generalizations. Note that in contrast to the finite horizon case considered in this paper, the process in \eqref{eq:infiniteth} starts in 0 and $\tau$ is positive.
\begin{theorem}\label{th:martin}
    The continuation set $C$ defined by the stopping problem \eqref{eq:sp12} fulfills the equation
    \begin{equation}\label{eq:main}
        0 = \int_{-\infty}^{0}\int_{C_s}e^{\textbf{c}\cdot \textbf{y} + \left(\frac{\norm{\textbf{c}}^2}{2}-r\right)s}\tilde h (\textbf{y})\d \textbf{y} \d s
    \end{equation}
    for all $\textbf{c}$ such that $\norm{\textbf{c}}>\sqrt{2r}$ and 
    $\{(t,-t\textbf{c})\mid t\leq 0\}\cap C$ is bounded.
\end{theorem}

\begin{proof}
    We pick $\textbf{c}\in\rr^{n}$ such that $\{(t,-t\textbf{c})\mid t\leq 0\}\cap C$ is bounded, then $(t,-t\textbf{c}) \in S$ for $t$ small enough. We set $\textbf{x} = -t\textbf{c}$ in \eqref{eq:pes}, divide by $p((t,-\textbf{c}t),(0,\textbf{0}))$ and obtain
    \begin{equation*}
        0 = \int_{t}^{0} \int_{C_s}e^{-rs}\tilde h (\textbf{y})\frac{p((t,-\textbf{c}t),(s,\textbf{y}))}{p((t,-\textbf{c}t),(0,\textbf{0}))} \d \textbf{y}\d s
    \end{equation*}
    for all $t$ small enough, if the integral exists.
    Taking limits we obtain
    \begin{equation}\label{eq:lim1}
        0 = \lim_{t\to -\infty}\int_{t}^{0} \int_{C_s}e^{-rs}\tilde h (\textbf{y})\frac{p((t,-\textbf{c}t),(s,\textbf{y}))}{p((t,-\textbf{c}t),(0,\textbf{0}))} \d \textbf{y} \d s.
    \end{equation}
    We analyze the quotient in  the integral. The kernel $p$ is the density of an $n$-dimensional normal distribution, so we have
    \begin{equation*}
    \begin{split}
        &\frac{p((t,-\textbf{c}t),(s,\textbf{y}))}{p((t,-\textbf{c}t),(0,\textbf{0}))}
       = \frac{(2\pi)^{-\frac{n}{2}} (s-t)^{-\frac{n}{2}} \exp\left(-\frac{\norm{-\textbf{c}t-\textbf{y}}^2}{2(s-t)}\right)}{(2\pi)^{-\frac{n}{2}}(-t)^{-\frac{n}{2}}\exp \left(-\frac{\norm{-\textbf{c}t}^2}{2(-t)}\right)} \\
        & \qquad\qquad = \left(\frac{-t}{s-t}\right)^{\frac{n}{2}} \exp\left(-\frac{\norm{-\textbf{c}t-\textbf{y}}^2}{2(s-t)} + \frac{\norm{-\textbf{c}t}^2}{2(-t)}\right)\\
        & \qquad\qquad  = \left(\frac{-t}{s-t}\right)^{\frac{n}{2}} \exp\left(\frac{\left(\norm{\textbf{c}}^2t^2 + 2\textbf{c}\cdot \textbf{y} t + \norm{\textbf{y}}^2\right)t + \norm{\textbf{c}}^2t^2(s-t)}{2(t^2-st)} \right)\\
        & \qquad\qquad  = \left(\frac{-t}{s-t}\right)^{\frac{n}{2}} \exp\left(\frac{2\textbf{c}\cdot \textbf{y} t^2 + t\norm{\textbf{y}}^2 +\norm{\textbf{c}}^2t^2s }{2(t^2-st)} \right),
    \end{split}
    \end{equation*}
    where $\cdot$ denotes the standard scalar product and $\norm{\cdot}$ the corresponding norm.
    We calculate the limit and have
\begin{align}\label{eq:martin}
	\lim_{t\to -\infty}\frac{p((t,-\textbf{c}t),(s,\textbf{y}))}{p((t,-\textbf{c}t),(0,\textbf{0}))} = \exp\left(\textbf{c}\cdot \textbf{y} + \frac{\norm{\textbf{c}}^2}{2}s \right).
	\end{align}
    We can pull the limit into the integral in \eqref{eq:lim1} (for a proof of that fact see Appendix \ref{sec:appendix})
    and obtain
    \begin{equation}
        0 = \int_{-\infty}^{0}\int_{C_s}e^{\textbf{c}\cdot \textbf{y} + \frac{\norm{\textbf{c}}^2}{2}s}e^{-rs}\tilde h (\textbf{y})\d \textbf{y} \d s
        =\int_{C}e^{\textbf{c}\cdot \textbf{y} + \frac{\norm{\textbf{c}}^2}{2}s}e^{-rs}\tilde h (\textbf{y})\d \textbf{y} \d s
    \end{equation}
    for all $\textbf{c}$ such that $\{(t,-t\textbf{c})\mid t\leq 0\}\cap C$ is bounded and the integral exists.
\end{proof}

The Fredholm type integral equation \eqref{eq:main} is highly non-linear. As far as we know, equations of this type have not been analyzed in the literature before.
To simplify equation \eqref{eq:main}, we define
    $C_{\infty}:=\bigcup_{t\leq 0} C_t$.
Note that in cases where the infinite time horizon problem
\begin{equation}\label{eq:infhorizon}
    V(\textbf{x}) = \sup_{0\leq\tau}\e_{\textbf{x}}[e^{-r\tau}h(W_\tau)]
\end{equation}
is solvable, $C_{\infty}$ is the solution to that problem under weak assumptions. 
Since the dimension of \eqref{eq:infhorizon} is reduced by one, $C_{\infty}$ is usually easier to find than $C$. This is well-known in the case $d=1$, see \cite{MR827892} or \cite{MR1999788}. An approach for $d>1$ is discussed in \cite{christensen16}.
We now look at some properties of $C_t$.
It is first easily seen that $C_t$ is decreasing in $t$. 
We also know that $C_0 = \set{\textbf{x}\mid \tilde h(\textbf{x})<0}$. We define the function
    \[d:C_\infty\setminus C_0 \to \rr_{\leq 0}, ~ \textbf{x}\mapsto \sup\set{t\leq0 \mid(t,\textbf{x})\in C}.\]
Since $C_t$ is decreasing, we see that $d$ defines $C$ via $C_t = d^{-1}((t,0])$.\\    
We can now simplify \eqref{eq:main} using Fubini's lemma 
\begin{alignat}{2}
    &&-\int_{-\infty}^{0} \,\int_{C_0}  e^{\frac{\|\textbf{c}\|^{2}}{2}s + \textbf{c}\cdot \textbf{y}-rs}\tilde h(\textbf{y})\d \textbf{y} \d s    
    &= \int_{-\infty}^{0}\int_{C_s\setminus C_0}  e^{\frac{\|\textbf{c}\|^{2}}{2}s + \textbf{c}\cdot y-rs}\tilde h(\textbf{y})\d \textbf{y} \d s \nonumber\\
    &\iff&
    -\int_{C_0} \,\int_{-\infty}^{0}  e^{\frac{\|\textbf{c}\|^{2}}{2}s + \textbf{c}\cdot \textbf{y}-rs}\tilde h(\textbf{y})\d s\d \textbf{y}   
    &=\int_{C_\infty\setminus C_0}\int_{-\infty}^{d(\textbf{y})}  e^{\frac{\|\textbf{c}\|^{2}}{2}s + \textbf{c}\cdot y-rs}\tilde h(\textbf{y})\d s \d \textbf{y}       \nonumber \\
     &\iff&
     -\int_{C_0}  e^{\textbf{c}\cdot \textbf{y}}\tilde h(\textbf{y})\d \textbf{y}
     &=\int_{C_{\infty}\setminus C_0} e^{\left(\frac{\|\textbf{c}\|^{2}}{2} - r\right)d(\textbf{y}) + \textbf{c}\cdot \textbf{y}}\tilde h(\textbf{y}) \d \textbf{y}   \label{eq:9}     
\end{alignat}
for all $\textbf{c}$ with $\norm{\textbf{c}} >\sqrt{2r}$ for which $\{(t,-t\textbf{c})\mid t\leq 0\} \cap C$ is bounded.
The integral on the left-hand side of \eqref{eq:9}
can be interpreted as a Laplace transform transform of $\tilde h$.

\begin{remark}
	The condition that $\{(t,-t\textbf{c})\mid t\leq 0\} \cap C$ is bounded does not seem easy to handle at the first moment. However, the previous considerations provide a simple sufficient condition:
	If $C_\infty$ is bounded, then $\{(t,-t\textbf{c})\mid t\leq 0\}\cap C$ is bounded for all $\textbf{c}\neq 0$ and \eqref{eq:main} holds for all $\textbf{c}$ with $\norm{\textbf{c}}>\sqrt{2r}$. An analogous result holds in one-sided cases, see the following discussion.
	\end{remark}

\subsection{The one-dimensional and one-sided case}\label{ssec:2}
%

In one dimension an important class of stopping problems is the class of problems with one-sided solutions. These have a continuation set that can be written as 
\[C = \set{(t,x)\mid x<b(t)}\]
for some function $b:\rr_{\leq0}\to\rr$. In the discounted setting, $b$ is the inverse function of $d$ constructed above.
The stopping boundary $b$ is decreasing and we define
    \[b_\infty:= \lim_{t\to-\infty} b(t) .\] 
If the corresponding infinite time horizon stopping problem is solvable, then $b_\infty<\infty$ and 
$C_\infty = (-\infty,b_\infty)$. Then, \eqref{eq:main} holds for all $c>\sqrt{2r}$. We can assume w.l.o.g.\ that $C_0 = (-\infty,0)$. Then, the integral transformation from \eqref{eq:9} is the Laplace transformation $\mathcal{L}$ and the representation can, for all $c>\sqrt{2r}$, be written as
\begin{equation}
   -\mathcal{L}\tilde h(c) = \int_{0}^{b_\infty} e^{\left(\frac{c^{2}}{2} - r\right)d(y) + cy}\tilde h(y) \d y .
\end{equation}

\subsection{Generalizations}\label{ssec:gen}
The Fredholm representation also holds for general payoff functions -- under some mild technical assumptions. Let $g:\rr_{\leq 0}\times \rr^{n} \to \rr$ be in $C^{1,2}$.
\begin{theorem}\label{th:main12}
    The continuation set $C$ defined by the stopping problem \eqref{eq:sp1} fulfills the equation
    \begin{equation}
        0 = \int_{-\infty}^{0}\int_{C_s}e^{\textbf{c}\cdot \textbf{y} + \frac{\norm{\textbf{c}}^2}{2}s}(-\g g)(s,\textbf{y}) \d \textbf{y} \d s
    \end{equation}
    for all $\textbf{c}$ for which the integral exists and there is an $\varepsilon >0$ such that $\{(t,-t(\textbf{c}+\textbf{a}))\mid t\leq 0, \norm{a}<\varepsilon\}\cap C$ is bounded.
\end{theorem}
he proof follows directly  as in the discounted case. We need the additional boundedness condition
for the application of the dominated convergence theorem. 

In the previous discussion, we limited ourselves to the case of Brownian motion as a driving process. The point at which this assumption entered centrally was the calculation of the limit \eqref{eq:martin}. The question arises for which more general Markov processes $X$ corresponding limits -- possibly along other paths than straight lines -- exist. This question is closely related to the classical Martin-boundary theory, see \cite{MR1814344} and \cite{chung2006markov}. More precisely, general potential theory provides that the boundary functions thus arising are exactly the harmonic functions of the space-time Markov process $(t,X_t)_{t\geq 0}$ and the equation analogous to \eqref{eq:main} reads as
\[    
	0 = \int_{-\infty}^{0}\int_{C_s}\kappa_c(s,\textbf{y})e^{-rs}\tilde h (\textbf{y})\d \textbf{y} \d s,
\]
where $\kappa_c$ denotes a family of harmonic functions, such that -- for a suitable index set $I$ -- $\{\kappa_c:c\in I\}$ denotes the Martin boundary. 
Thus, whether the previous procedure can be fruitfully applied is determined by whether this Martin boundary can be found explicitly and is rich enough. This question has been discussed in the literature for different example classes. Reference is made, for example, to the case of geometric Brownian motion in \cite{christensen2013riesz} and to \cite{MR606005} for Cauchy-, $d$-dimensional Bessel- and Poisson processes. 
In the latter article it is shown that in the case of an underlying Cauchy process the Martin-boundary consists only of the constant functions, so that the approach presented here is not suitable for characterizing the stopping boundary for general processes. Because of this and since the arguments used in the following sections are process-specific, we continue to focus on the\ case of Brownian motion and leave the study of other problem classes for future research.

\subsection{An example}
We conclude the section with an example to illustrate how the Fredholm representation can be used to tackle explicitly solvable problems.

\begin{example}\label{ex:stadje}
    Let $n=1$ , $h(x) = \frac{x^3}{3}$ and $r=0$. We have $\tilde h(x) = -x$.
    Motivated by Brownian scaling, we make the ansatz that the continuation set is of the form 
        \[C = \{(t,x) \mid x < \alpha \sqrt{-t}\}\]
    for some $\alpha >0$. We have $-\g g(t,x) = -x$, plugging this into \eqref{eq:main} we get
    \begin{align*}
        0 & = \int_{-\infty}^{0}\int_{-\infty}^{\alpha \sqrt{-s}} -x e^{c y + \frac{c^2}{2}s}\d y \d s
         = \frac{1}{c^2}\int_{-\infty}^{0}(1-\alpha c \sqrt{-s})e^{\frac{c^2}{2}s+ \alpha c \sqrt{-s}} \d s\\
        & = \frac{1}{c^2}\left( -2\alpha^3 \frac{\Phi(\alpha)}{\phi(\alpha)} +(1-\alpha^2)  \right) 
    \end{align*}
    where $\Phi$ and $\phi$ denote CDF and PDF of a standard normal distribution, resp. 
    For $c>0$ the equation can be reformulated as $\alpha^3 \Phi(\alpha)  = (1-\alpha^2)\phi(\alpha)$ which has a unique positive solution at $\alpha_1 \approx 0.638833$.
    
    By the uniqueness results in Section \ref{sec:u}, we will see that indeed the continuation set of the problem is
        \[C = \{(t,x) \mid x < \alpha_1 \sqrt{-t}\}.\]
    We remark that the \emph{commodity sales} problem described by Stadje in \cite{stadje87}
        $  V(t,x) = \sup_{t\leq\tau\leq0}\e_{(t,x)}[\tau W_\tau]$
    leads to the same Fredholm integral representation as our example and hence has the same solution.
\end{example}

\section{Limit behavior}\label{sec:3}
In this section we show how to use the Fredholm representation to derive analytical properties of $C$. In particular, we study the limit behavior of $C_t$ for $t\to 0$. We do this rigorously for the one-dimensional, one-sided discounted case in Subsection \ref{ssec:1} and give heuristic arguments for multidimensional stopping problems in Subsection \ref{ssec:n}. Our technique builds on Theorem \ref{th:martin} and is independent of the uniqueness discussed in Section \ref{sec:u}.
In special cases, the results are known, but to our knowledge they are new in the generality given below.

\subsection{The one-dimensional case}\label{ssec:1}
In the setting of Section \ref{ssec:2} assume w.l.o.g. that $C_0 = (-\infty,0)$, so our main equation reads
 \begin{equation}\label{eq:10}
   -\mathcal{L}\tilde h(c) 
   = \int_{0}^{b_\infty} e^{\left(\frac{c^{2}}{2} - r\right)d(y) + cy}\tilde h(y) \d y 
\end{equation}
for all $c>\sqrt{2r}$.
  We assume for now that $\tilde h$ is continuous. Then, since $\tilde h(x) \leq 0$ for $x<0$ and $\tilde h(x) \geq 0$ for $x>0$, we have $\tilde h(0) =0 $. Central for the limit behavior is the degree of $h$ close to 0. The most common case is that $\tilde h$ is approximately linear, i.e., $\tilde h(x) = m x + o(x)$ for some $m>0$ in a neighborhood of 0. We will stick to that case for sake of clarity and refer to Remark \ref{rem:nonlin} below for generalizations.
%
 For the left-hand side of \eqref{eq:10} we have
   $\lim_{c\to\infty}-c^{2}\mathcal{L}\tilde h(c)   =m$
 so we obtain
    \[m = \lim_{c\to\infty}c^2\int_{0}^{b_\infty} e^{\left(\frac{c^{2}}{2} - r\right)d(y) + cy}\tilde h(y) \d y. \]
The following two lemmata show that only a small neighborhood of 0 and the limit behavior of $\tilde h(x)$ for $x\to 0$ are relevant for the limit behavior of $C_t$.
    
\begin{lemma}\label{lem:eps1}
    For all $\varepsilon > 0$ it holds
        \[\lim_{c\to\infty}c^2\int_{0}^{b_\infty} e^{\left(\frac{c^{2}}{2} - r\right)d(y) + cy}\tilde h(y) \d y
        =\lim_{c\to\infty}c^2\int_{0}^{\varepsilon} e^{\left(\frac{c^{2}}{2} - r\right)d(y) + cy}\tilde h(y) \d y.\]
\end{lemma}

\begin{proof}
    We know from the construction of the stopping problem that
    $d$ is non-increasing, $d(0) = 0$ and $d(x)<0$ for all $x > 0$. Then $d(\varepsilon) < 0$ and $d(x)\leq d(\varepsilon)$ for all $x \geq \varepsilon$. We have
    \begin{equation*}
    \begin{split}
       \left|\lim_{c\to\infty}c^2\int_{\varepsilon}^{b_\infty} e^{\left(\frac{c^{2}}{2} - r\right)d(y) + cy}\tilde h(y) \d y\right|
        &\leq  C_1 \lim_{c\to\infty}c^2\int_{\varepsilon}^{b_\infty} e^{\left(\frac{c^{2}}{2} - r\right)d(\varepsilon) + cy}\d y\\
         &\leq  C_2 \lim_{c\to\infty}c^2 e^{\frac{c^{2}}{2} d(\varepsilon)}    \int_{\varepsilon}^{b_\infty} e^{cy}\d y \\
         &=  C_2 \lim_{c\to\infty}c^2 e^{\frac{c^{2}}{2} d(\varepsilon)}   \frac{1}{c}(e^{cy}-1) = 0,
    \end{split}
    \end{equation*}
    where $C_1$ and $C_2$ are some constants. The result follows because of $[0,b_\infty] = [0,\varepsilon)\cup [\varepsilon,b_{\infty}]$ and the linearity of the integral.
\end{proof}

\begin{lemma}\label{lem:eps2}
    If $\tilde h(x) = mx + o(x)$ in a neighborhood of 0, then
        \[\lim_{c\to\infty}c^2\int_{0}^{b_\infty} e^{\left(\frac{c^{2}}{2} - r\right)d(y) + cy}\tilde h(y) \d y
        =\lim_{c\to\infty}c^2\int_{0}^{b_{\infty}} e^{\left(\frac{c^{2}}{2} - r\right)d(y) + cy}m y \d y.\]
\end{lemma}

\begin{proof}
    For all $\delta >0$ there exists $\varepsilon_{\delta} >0$ such that for all $y\in[0,\varepsilon_{\delta}]$ we have
        \[|h(y)-my|<\delta y.\]
    It follows with Lemma \ref{lem:eps1} that
    \begin{align*}
        &\left|\lim_{c\to\infty}c^2\int_{0}^{b_\infty} e^{\left(\frac{c^{2}}{2} - r\right)d(y) + cy}\tilde h(y) \d y
        -\lim_{c\to\infty}c^2\int_{0}^{b_{\infty}} e^{\left(\frac{c^{2}}{2} - r\right)d(y) + cy}m y \d y\right|\\
        & \quad=  \left|\lim_{c\to\infty}c^2\int_{0}^{\varepsilon_{\delta}} e^{\left(\frac{c^{2}}{2} - r\right)d(y) + cy}\left(\tilde h(y)-my\right) \d y\right|\\
        & \quad\leq \left|\lim_{c\to\infty}c^2\int_{0}^{\varepsilon_{\delta}} e^{\left(\frac{c^{2}}{2} - r\right)d(y) + cy}\delta y \d y\right|
        \overset{\delta \to 0}\to 0.
    \end{align*}
\end{proof}

\noindent We can now state the main theorems of this section.

\begin{theorem}\label{cor:lim}
    If $\lim_{x\searrow 0} \frac{d(x)}{x^2}$ exists in $\bar \rr = [-\infty,\infty]$, then
        \[\lim_{x\searrow 0} \frac{d(x)}{x^2} = -B_1,\]
    i.e.,
        $d(x) = -B_1x^2 + o(x^2)$, 
    where
        $B_1 \approx 2.4503$
    is the unique solution to 
    \[1 = \int_{0}^{\infty} z e^{-B_1\frac{z^2}{2} + z}\d z.\]
\end{theorem}

\begin{proof}
    Assume that $\lim_{x\searrow 0} \frac{d(x)}{x^2}$ exists in $\bar \rr$. We first show that $d$ goes to 0 in quadratic order. Let us assume that $d$ is of lower order, i.e.,
        \[\lim_{x\searrow 0} \frac{d(x)}{-x^2} = \infty.\]
    For all $M< 0$ there exists $\delta>0$ such that $d(x)<Mx^2$ for all $x\leq \delta$.
    By Lemma \ref{lem:eps1} we have $m = \lim_{c\to\infty} c^2 \int_{0}^{\delta}  e^{\left(\frac{c^{2}}{2} - r\right)d(y) + cy} my\d y$, i.e., 
    \begin{align*}
 1& =\lim_{c\to\infty} c^2 \int_{0}^{\delta} y e^{\left(\frac{c^{2}}{2} - r\right)d(y) + cy} \d y.
    \end{align*}
    We substitute $z = cy$ and obtain
    \begin{equation}\label{eq:z2}
     \begin{split}
        1 &= \lim_{c\to\infty} c \int_{0}^{c\delta} \frac{z}{c} e^{\left(\frac{c^{2}}{2} - r\right)d\left(\frac{z}{c}\right) + z}\d z  
            = \lim_{c\to\infty} \int_{0}^{c\delta} z e^{\left(\frac{c^{2}}{2} - r\right)d\left(\frac{z}{c}\right) + z}\d z\\
        & \leq \lim_{c\to\infty} \int_{0}^{c\delta} z e^{\left(\frac{c^{2}}{2} - r\right)M\frac{z^2}{c^2}+ z}\d z 
        =\lim_{c\to\infty} \int_{0}^{c\delta} z e^{M\frac{z^2}{2}+ z}\d z
        =:F(M).
     \end{split}
    \end{equation}
    We see that $F(M)\to 0$ as $M\to -\infty$. This is a contradiction, since $M$ can be chosen arbitrarily. 

    Let us now assume that $d$ is of higher order than $x^2$, i.e.,
    \[\lim_{x\searrow 0} \frac{d(x)}{-x^2} = 0.\]
    For all $M < 0$ there exists $\delta'>0$ such that $d(x)>Mx^2$ for all $x\leq \delta'$. Note that $d$ is non-positive.
    From \eqref{eq:z2} we see that
    \begin{align}
        1 &= \lim_{c\to\infty} \int_{0}^{c\delta'} z e^{\left(\frac{c^{2}}{2} - r\right)d\left(\frac{z}{c}\right) + z}\d z
         \geq \lim_{c\to\infty} \int_{0}^{c\delta'} z e^{\left(\frac{c^{2}}{2} - r\right)M\frac{z^2}{c^2}+ z}\d z\nonumber \\
        &=\lim_{c\to\infty} \int_{0}^{c\delta'} z e^{M\frac{z^2}{2}+ z}\d z
        =:\tilde F(M).\nonumber
    \end{align}
    We see that $\tilde F(M)\to \infty$ as $M\to 0$. This is a contradiction, since $M$ can be chosen arbitrarily.
    We have shown that $d(x) = -Bx^2 + o(x^2)$ and want to calculate $B$. From \eqref{eq:z2} we have
    \begin{align}
        1 &= \lim_{c\to\infty} \int_{0}^{c\delta'} z e^{\left(\frac{c^{2}}{2} - r\right)d\left(\frac{z}{c}\right) + z}\d z = \lim_{c\to\infty} \int_{0}^{c\delta'} z e^{\left(-\frac{c^{2}}{2} - r\right)\left(B\frac{z^2}{c^2} + o(\frac{z^2}{c^2})\right) + z}\d z\nonumber\\
        &= \int_{0}^{\infty} z e^{-B\frac{z^2}{2} + z}\d z.\nonumber
    \end{align}
    Since the integral is finite, we can solve it for $B$ and obtain
       $B = B_1 \approx 2.4503$.
\end{proof}

\noindent The following theorem generalizes Theorem \ref{cor:lim} and does not assume the existence of $\lim_{x\searrow 0} \frac{d(x)}{x^2}$.

\begin{theorem}\label{th:d1}
    It holds that 
\begin{align*}
	-B_1\leq\limsup_{x\searrow 0} \frac{d(x)}{x^2}< 0\mbox{ and } -\infty<\liminf_{x\searrow 0} \frac{d(x)}{x^2}\leq -B_1\end{align*}
    where $B_1\approx 2.4503$ is defined above.
\end{theorem}
The proof of the theorem is basically a refined version of the proof of Theorem \ref{cor:lim}. It can be found in \cite{phd_simon}.
 
\begin{remark}
    We note that $B$ does not depend on $m$ or $r$, i.e., it is a universal constant for stopping problems with the same order of $\tilde h$ in 0.
    This is however not surprising. Multiplying the payoff function $g$ by a constant does not affect $C$, so $m$ should not play a role. We could achieve a change of $r$ by Brownian scaling, i.e., scaling $t$ by a factor $a$ and $y$ by $\sqrt{a}$ and multiplying $g$ by a constant. A quadratic function such as $d$ is invariant to that.
\end{remark}
 
\begin{remark}\label{rem:nonlin}
    If $\tilde h$ is not linear in 0 but rather
        \[\tilde h(x) = \begin{cases}
                    m x^{\beta} + o(x^{\beta}) & \mbox{if }x\geq 0\\
                    -m' |x|^{\beta} + o(|x|^{\beta})& \mbox{if }x< 0
                    \end{cases}
        \]
    for $m,m',\beta>0$, then the same procedure as above yields
        \[\frac{m'}{m}\Gamma(\beta + 1) = \int_{0}^{\infty} z^{\beta} e^{- B_{\beta}\frac{z^2}{2} + z}\d z\]
    which can be solved for $ B_{\beta}$. For $m' = m$, some values are:
    $B_0 \approx 3.9084$, $B_{\frac{1}{2}} \approx 3.0133$, $B_2 \approx 1.7814$ and $B_3 \approx 1.3984$.
    With the implicit function theorem we see that $B_\beta$ is continuous as a function of $\beta$, hence other choices for $\tilde h$ lead to the same limit behavior, e.g., $\tilde h(x) = \frac{x}{\log(|x|)}$ yields the same limit behavior for $C_t$ as $\tilde h(x) = x$.
\end{remark}
 
\begin{remark}
    If we change our perspective again and state the result for the function $b = d^{-1}$ we get (if the limit exists)
        $b(t) = \alpha \sqrt{t} + o(\sqrt{t})$
    for $t\to 0$, where
        $\alpha = \frac{1}{\sqrt{B_1}} \approx 0.6388$
    is the unique solution to $\alpha^3 \Phi(\alpha) = (1-\alpha^2)\varphi(\alpha)$. This $\alpha$ is the same as in Example \ref{ex:stadje}, which is not surprising since the example matches our setting. The same limit behavior was earlier derived for the exercise boundary of an American put with large continuous dividend payments. We treat that case in the following Example \ref{ex:option}. For related problems, see \cite{wilmott98}, \cite{lamberton03} and \cite{lai05}.
\end{remark}
 
\begin{example}[American put with large dividend]\label{ex:option}
    We consider an American put option in the Black-Scholes model with interest rate $r$ and continuously paid dividend rate $q$. The problem can be reduced to the following \emph{canonical} optimal stopping problem 
    (see \cite{lai99})
        \[V(t,x) = \sup_{t\leq\tau\leq0}\e_{(t,x)}[g(\tau,W_{\tau})]\]
    where 
        $g(t,x) = e^{-\rho t}\left(1-e^{x + (\rho - \theta \rho -\frac{1}{2})t}\right)^{+}$
    with $\rho = \frac{r}{\sigma^2}$ and $\theta = \frac{q}{r}$. We consider the case $q>r$, i.e., $\theta<1$. For $x < 0$ we have
    \begin{equation*}\label{eq:put}
        -\g g(t,x) = e^{-\rho t}\left(\rho - \theta \rho e^{x + (\rho - \theta \rho -\frac{1}{2})t}\right).
    \end{equation*}
    $\g g(0,x)$ has a unique negative root in $x_0 = \log(\frac{1}{\theta}) = \log(\frac{r}{q})$ and $-\g g(0,x)\leq 0$ for $x \geq x_0$, so we have
        $C_0 = \left(\log\left(\frac{r}{q}\right),\infty\right).$
    By the variable transformation $z = x + \log\left(\frac{r}{q}\right) + (\rho - \theta \rho -\frac{1}{2})t$ we get the stopping problem 
        \[V(t,z) = \sup_{t\leq\tau\leq0}\e_{(t,z)}[e^{-\rho \tau} h(X_{\tau})]\]
    where $h(z) = (\rho - \rho e^{z})^{+}$ and $X$ is a Brownian motion with drift. Keeping Remark \ref{rem:drift} in mind, the problem now matches the assumptions of Theorem \ref{th:d1}, so we see that the continuation set of the \emph{canonical} problem is $C = \{(t,x)\mid t<d(x)\}$ with
        \[ \lim_{x\nearrow \log\left(\frac{r}{q}\right)} \frac{d(x)}{x^2} = -B.\]

    This result matches the limit behavior derived for the American put (and call) in \cite{lamberton03}. Note that for $q\leq r$ Theorem \ref{th:d1} is no longer directly applicable since $g(0,x)$ has a kink in 0, i.e., right at the boundary of $C_0$, so $\tilde h$ has a pointmass in 0 and is not asymptotically symmetric in 0. It seems possible to use the methods described above to analyze these cases as well. We, however, run into more technicalities.
\end{example}

\subsection{The multi-dimensional case}\label{ssec:n}
We now consider the $n$-dimensional case. We will not give rigorous proofs, but explain rather heuristically how the methods can be extended.

We change the notation slightly. For notational convenience, we substitute the parameter $\textbf{c}$ by $a\textbf{c}$, where $\textbf{c}\in S^{n-1}$ and $a\in \rr^{+}$. Our equation now reads
\begin{equation}\label{eq:a}
 -\int_{C_0}  e^{a \textbf{c}\cdot \textbf{y}}\tilde h(\textbf{y})\d \textbf{y}
     =\int_{C_{\infty}\setminus C_0} e^{\left(\frac{a^{2}}{2} - r\right)d(\textbf{y}) + a\textbf{c}\cdot \textbf{y}}\tilde h(\textbf{y}) \d \textbf{y}.
\end{equation}
We assume for now that $C_0$ is strictly convex. Then for every $\textbf{c}$ for which $\{a\textbf{c}\mid a\in\rr^{+}\}\cap C_\infty$ is bounded, there is a unique point $\textbf{x}_c\in\partial C_0$ realizing $\sup\{\textbf{x}\cdot\textbf{c}\mid\textbf{x}\in \bar C_0\}$. An illustration can be found in Figure  \ref{fig:transf1}.
The underlying stopping problem is invariant under rotation and translation, so we can reset the coordinate system such that $\textbf{x}_c$ lies in the origin and $\textbf{c} =(1,0,\ldots,0)$, see Figure \ref{fig:transf2}.
\begin{figure}
\centering
\begin{subfigure}{.5\textwidth}
  \centering
  \includegraphics[width=.9\linewidth]{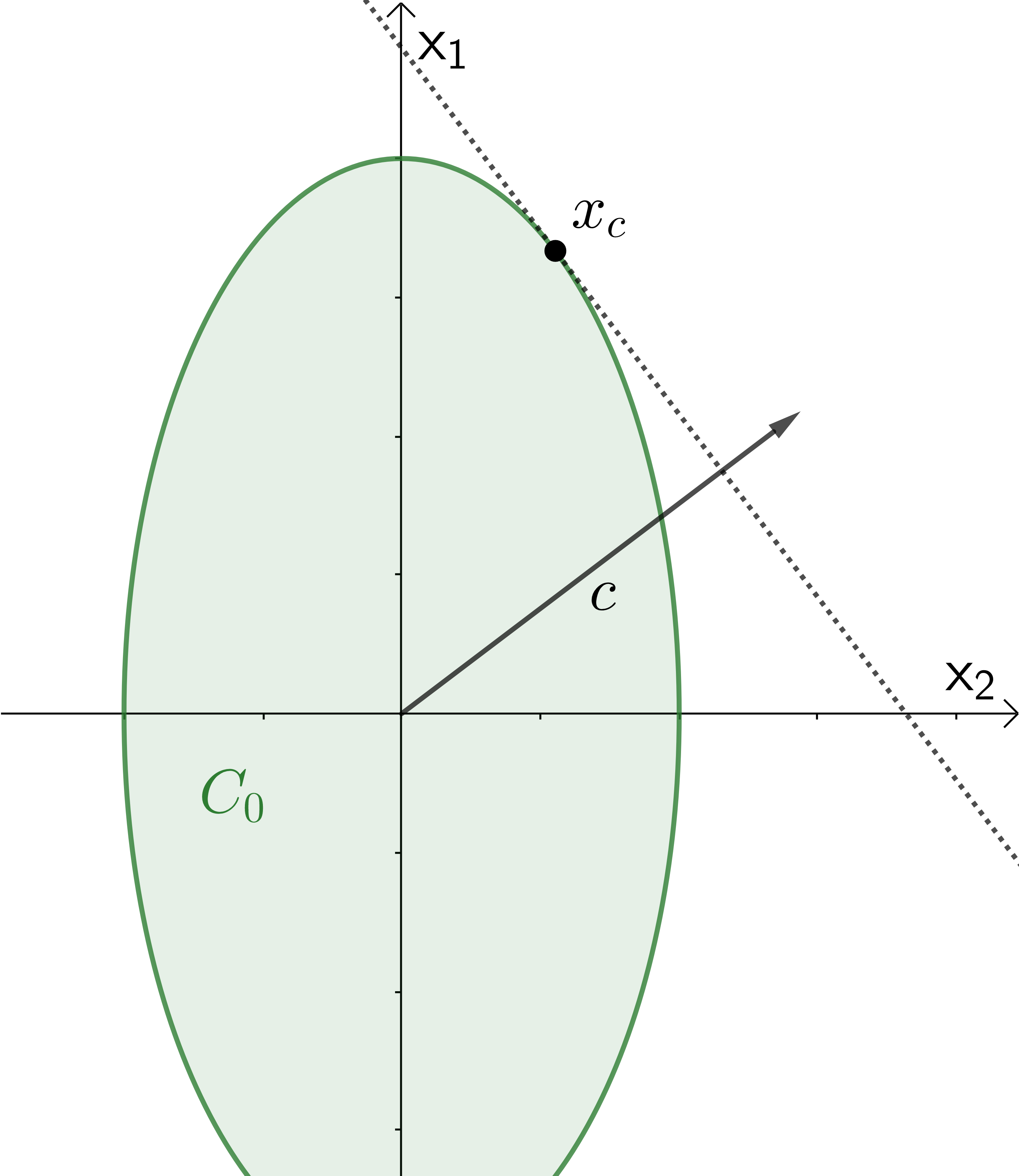}
  \caption{Original problem}
  \label{fig:transf1}
\end{subfigure}%
\begin{subfigure}{.5\textwidth}
  \centering
  \includegraphics[width=.9\linewidth]{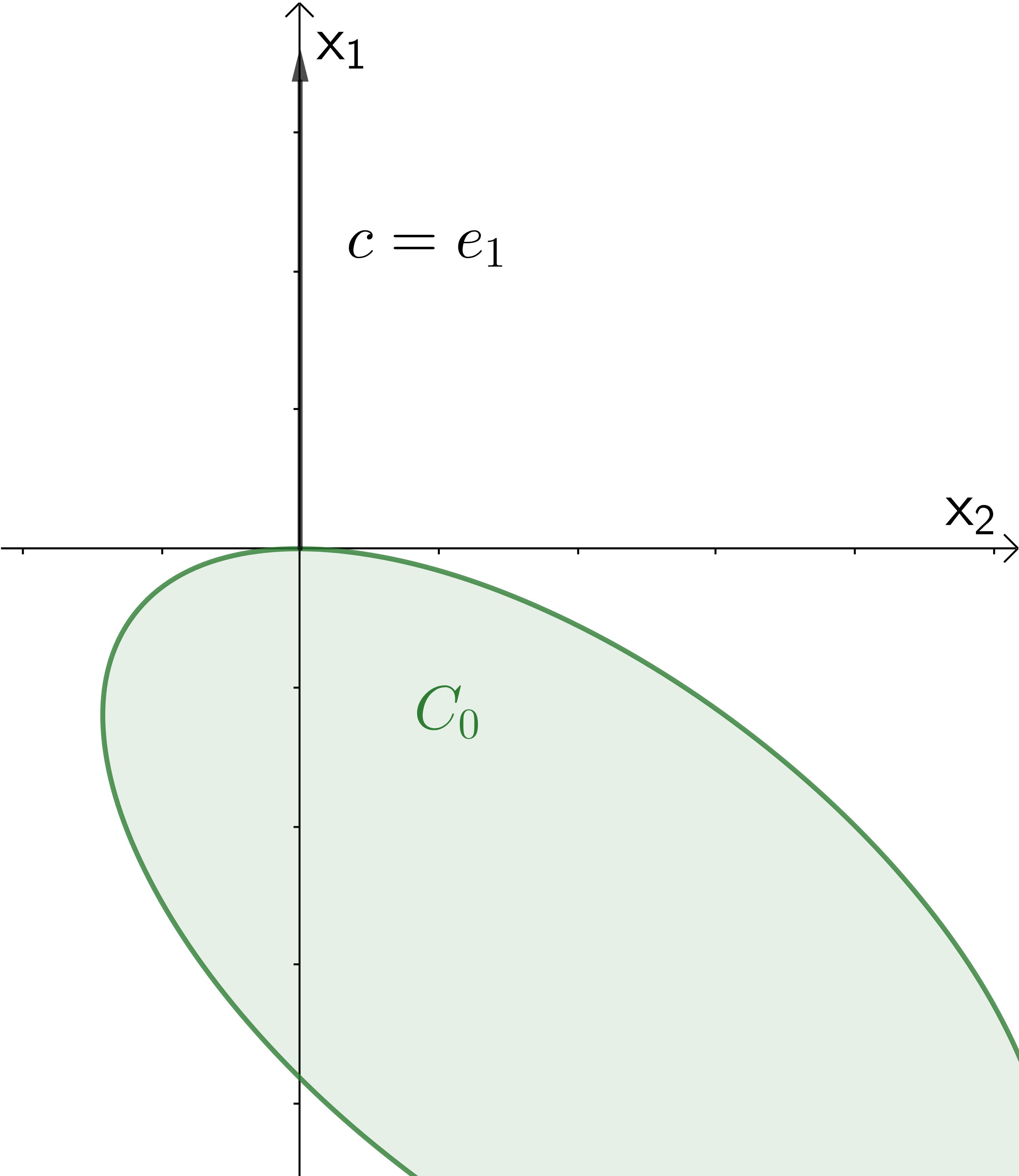}
  \caption{Rotated and translated}
  \label{fig:transf2}
\end{subfigure}
\caption{The setting in the two-dimensional case}
\label{fig:test}
\end{figure}
Now \eqref{eq:a} reads
\begin{equation}\label{eq:a2}
    -\int_{C_0}  e^{a y_1}\tilde h(\textbf{y})\d \textbf{y}
    =\int_{C_{\infty}\setminus C_0} e^{\left(\frac{a^{2}}{2} - r\right)d(\textbf{y}) + ay_1}\tilde h(\textbf{y}) \d \textbf{y}.
\end{equation}
We assume that $\partial C_0$ has positive curvature, so it can be approximated in a neighborhood of 0 by $y_1 = -\sum_{i=2}^{n}\alpha_i y_i^2$ with constants $\alpha_i >0$. The constructions are similar and lead to the same result if we consider different approximations such as a (piecewise) constant boundary, corners or $x_1 \approx -\sum\alpha_i \operatorname{sgn}(x_i)|x_i|^{p}$, $p\in \rr_{\geq 0}$, in general.

\paragraph{Two-dimensional}
For simplicity let us look at the case $n=2$. We set $\gamma = \gamma_2$ and calculate the limit of the left-hand side of \eqref{eq:a2}:
\begin{align*}
    &\lim_{a\to \infty}-a^{\frac{5}{2}}\int_{C_0}  e^{a y_1}\tilde h(\textbf{y})\d \textbf{y}\\
    &\quad = \lim_{a\to \infty}- a^{\frac{5}{2}}\int_{-\infty}^{0} e^{a y_1} 
    \int_{-\sqrt{\frac{y_1}{\gamma}}}^{\sqrt{\frac{y_1}{\gamma}}} m(y_1-\gamma y_2^2)\d y_2 \d y_1
    =  m \sqrt{\frac{\pi}{\gamma}}
\end{align*}
For the right-hand side of \eqref{eq:a2} we see that only neighborhoods of 0 are asymptotically relevant. It follows -- substituting $z_1 = a y_1$ and $z_2 = \sqrt{a} y_2$ and setting  set $d(x_1,x_2) = -B(x_1+\gamma x_2^2)^2$  -- that
\begin{align*}
    & \lim_{a\to \infty}a^{\frac{5}{2}}\int_{C_{\infty}\setminus C_0} e^{\left(\frac{a^{2}}{2} - r\right)d(\textbf{y}) + ay_1}\tilde h(\textbf{y}) \d \textbf{y} \\
    &\quad = \lim_{a\to \infty}a^{\frac{5}{2}}\int_{-\infty}^{\infty} \int_{-\gamma y_2^2}^{\infty}  m(y_1+\gamma y_2^2) e^{\left(\frac{a^{2}}{2} - r\right)d(\textbf{y}) + ay_1} \d y_1 \d y_2\\
    &\quad = \lim_{a\to \infty}-\int_{-\infty}^{\infty} \int_{-a\gamma z_2^2}^{\infty}  m(z_1+\gamma z_2^2) e^{\left(\frac{a^{2}}{2} - r\right)d\left(\frac{z_1}{a},\frac{z_2}{\sqrt{a}}\right) + z_1} \d z_1 \d z_2\\
    &\quad =\lim_{a\to \infty}\int_{-\infty}^{\infty} \int_{0}^{\infty}  m z_1 e^{\left(\frac{a^{2}}{2} - r\right)\frac{-B}{a^2}z_1^2 + z_1 - \gamma z_2^2} \d z_1 \d z_2\\
    &\quad = \int_{-\infty}^{\infty} \int_{0}^{\infty}  m z_1 e^{\frac{-B}{2}z_1^2 + z_1 - \gamma z_2^2} \d z_1 \d z_2\\
    &\quad =m\sqrt{\frac{\pi}{\gamma}}\int_{0}^{\infty}z_1 e^{\frac{-B}{2}z_1^2 + z_1}\d z_1.
\end{align*}
Comparing the two sides we get
\begin{align*}
    m \sqrt{\frac{\pi}{\gamma}} & = m \sqrt{\frac{\pi}{\gamma}}\int_{0}^{\infty}z_1 e^{\frac{-B}{2}z_1^2 + z_1}\d z_1
    \mbox{, i.e., } \quad 1  = \int_{0}^{\infty}z_1 e^{\frac{-B}{2}z_1^2 + z_1}\d z_1. 
\end{align*} 
This is the same equation as in the one-dimensional case, which is again solved by $B_1$.

In higher dimensions the method works in exactly the same way.
These results have a nice interpretation that becomes clear when we write it in terms of distance $r= r(\textbf{y}) := \operatorname{dist}(\textbf{y},C_{0})$ to $C_0$. 
Now for $\textbf{y}\notin C_0$ close to $\partial C_0$ we have 
    \[d(\textbf{y}) = -B_1r^2 + o(r^2).\]
So for $t$ close to 0 we have that $C_t$ is approximately $C_0$ inflated by the constant amount
    $\frac{1}{\sqrt{B_1}}\sqrt{t}$
equally on each point.

\subsection{Possible generalizations}
It seems possible that the results on limit behavior are far from the most general setting that can be addressed by our method, so let us briefly discuss how further generalizations might look like:

In Lemma \ref{lem:eps1} and \ref{lem:eps2} we have shown that in the discounted case for \emph{regular} $\tilde h$ only the limit behavior of $\tilde h$ is relevant for the limit behavior of $d$ or $b$. This should hold true for general stopping problems where $\g g$ is sufficiently \emph{regular} close to $\partial C_0$. 

Consider a Brownian motion $X$ with volatility $\sigma^2$ and a stopping problem that matches (except the different driving process $X$) the assumptions of Theorem \ref{cor:lim} and has a stopping boundary $d$. We can scale the problem by $\frac{1}{\sigma}$ to obtain a problem for a standard Brownian motion $W$. For its stopping boundary $\tilde d$ we get $\tilde d(y) = d(\sigma y) = -B_1y^2 + o(y^2)$ in a neighborhood of 0. So we have   
$d(x)  = \frac{-B_1}{\sigma^2}x^2 + o(x^2)$
in a neighborhood of 0. Since only the local behavior is relevant, similar should be true for more general diffusions. 

In the $n$-dimensional setting we assumed that $C_0$ is convex. Since limit behavior for $t\to 0$ is a local phenomenon, it would be surprising if the global shape of $C_0$ was relevant for it. So we expect the described limit behavior to hold for all parts of $C_0$ that are locally convex. It would be surprising as well if the limit behavior for locally concave parts would be fundamentally different, but the Fredholm representations does not seem suitable to analyze these areas.

It is possible to analyze the limit behavior of $C_t$ for $t\to -\infty$ (or limit of $d(y)$ as $y\to b_\infty$ resp.) by looking at $c\searrow \sqrt{2r}$. This leads to an exponential lower bound for $b$. More precisely, for $t\to -\infty$ it is not difficult to show that
$b(t) \geq b_{\infty} - M e^{t},$
for some $M>0$. We omit the details here.

\section{Uniqueness}\label{sec:u}
We show that $C$ is defined uniquely by the Fredholm representation in the one-dimensional and one-sided case described in Subsection \ref{ssec:2}. The integral equation can not have a unique solution for areas where the integrand equals zero, so we have to additionally assume that $\tilde h =  0$ holds true only on a null set.
We list all assumptions that we use for the uniqueness results.
\begin{assumptions}\label{as:1}~
    \begin{itemize}
        \item The payoff function is of the form $e^{-rt}h(x)$
        with $r>0$.
        \item The Volterra integral equation \eqref{eq:pes} holds.
        \item The integral 
        \[\int_{-\infty}^{b_\infty} e^{\left(\frac{c^{2}}{2} - r\right)d(y) + cy}\tilde h(y) \d y \]
        exists.
        \item $\lambda(\{y\mid\tilde h(y) =0\})=0$ where $\lambda$ denotes the Lebesgue measure.
        \item The solution is one-sided, i.e., there exists a bounded function 
            $b: \rr_{-}\to\rr$
        such that $C = \{(t,x)\mid x< b(t)\}$. As above, we set $d = b^{-1}$.
    \end{itemize}
\end{assumptions}

In this setting, we show that \eqref{eq:main} defines $b$ uniquely in the class of continuous and monotone functions.
The ansatz we use is not limited to this special case, in principle it can be used on multidimensional problems in a more general setting. This, however, brings a lot of  technicalities, so we stick to the case described. 
Note that, as pointed out in Section \ref{sec:2}, we do not make any direct assumptions on the differentiability of $h$. We only need that \eqref{eq:pes} holds. For notational convenience, we will state the proofs using $\tilde h$ as a function. The proofs, however, work in the same way if we understand $\tilde h$ as a measure and evaluate integrals of the form $\int_{-\infty}^{b_\infty} e^{\left(\frac{c^{2}}{2} - r\right)d(y) + cy} \d \hat h(y)$.  In this way, standard American options, for example, are also included in the analysis, see \cite{lai99}.

\begin{remark}
    Our proof is inspired by techniques introduced by
    Bruni and Koch in their 1985 work on the identifiability of mixtures of Gaussian densities, see \cite{bruni85}.
    To see the connection, one can interpret the Fredholm representation as a mixture of Gaussian densities. Indeed, we have
        \[e^{cy+\frac{1}{2}c^2s} = A(s,y)\varphi\left(\frac{y}{-s},\frac{1}{-s}\right)(c),\]
    where $\varphi\left(\frac{y}{-s},\frac{1}{-s}\right)$ is the density function of a normal distribution with mean $\frac{y}{-s}$ and variance $\frac{1}{-s}$ and $A(s,y)$ is a positive function not depending on $c$. Bruni and Koch analyze mixtures of the form 
        \[f(c) = \int_{D}\varphi\big(\lambda_1(x),\lambda_2(x)\big)(c)\d \mu(x)\]
    where $D$ is compact, the mean $\lambda_1$ and the variance $\lambda_2$ are $C^1$ on $D$ and bounded from above, $\lambda_2$ is additionally bounded away from 0. Under some additional assumptions they show that for given $f$ the measure $\mu$ is defined uniquely by the equation, i.e., in statistical terms, it is identifiable. They state that compactness of $D$ is essential for their result. In our setting, however, the area of integration is not compact and mean and variance are not bounded, so we cannot directly use their results. Theorem \ref{th:eind} can be viewed as a result about identifiability in a special case of a non-compact $D$ and unbounded $\lambda_1$ and $\lambda_2$.
    We will keep in mind that we work with Gaussian mixtures with variable $c$, but we stick to the notation $e^{cy+\frac{1}{2}c^2s}$.
\end{remark}

\noindent The proof consists of the following steps:
\begin{enumerate}
    \item Show that the area with the highest variances, i.e., the largest $t$, govern the limits $c\to\infty$, see Lemma \ref{lem:lim}.
    \item Use 1.\ to show that, if the Fredholm representation holds true for two different functions $b$ and $b'$, then this still holds true if the integrand is multiplied by a polynomial $q(s,y)$, see Lemma \ref{lem:1}.
    \item Show that the polynomials (in fact we will use Laguerre exponential polynomials) lie dense in $L^{2}(\rr^{2})$, see Theorem \ref{th:eind}.
    \item Represent $b$ and $b'$ as part of measures $\mu$ and $\mu'$ on $\rr^{2}$ and use 2.\ and 3.\ to show that $\mu = \mu'$, see Theorem \ref{th:eind}.
\end{enumerate}

We recall that \eqref{eq:main} can be written as
\begin{equation}\label{eq:main22}
     -\frac{1}{\frac{c^2}{2}-r} \mathcal{L}\tilde h (c)
     = \int_{-\infty}^{0}\int_{0}^{b(s)}e^{cy+\frac{1}{2}c^2s-rs}\tilde h(y)\d y \d s. 
\end{equation}
The integrand on the right-hand side is non-negative, the left-hand side is the known function $f(c)$. 
We want to show uniqueness of $b$ in the class of positive continuous and monotone functions. 
A short notation will come very handy for the latter proofs, so for 
 measurable $J\subset \rr_{\leq 0}$ and $n,m\in\nn_{0}$ we set
\begin{align*}
    & I^{b}_{J}(n,m)(c) := \int_{J}\int_{0}^{b(s)}(-s +1)^{n}y^{m}e^{cy+\frac{1}{2}c^2s-rs}\tilde h(y)\d y \d s,\\
    & I^{b}(n,m)(c): = I^{b}_{(-\infty,0]}(n,m)(c),\\
    & I^{b}_{J}(c): = I^{b}_{J}(0,0)(c),
\end{align*}
and for two different continuous functions $b$ and $b'$ we define
\begin{align*}
   & D_{J}(n,m)(c) := I^{b}_{J}(n,m)(c) - I^{b'}_{J}(n,m)(c)\\ 
   & D(n,m)(c) := D_{(-\infty,0]}(n,m)(c).
\end{align*}

The following lemma states that only the parts with the largest variance, i.e., the largest $t$ values in our setting, play a role for the limits $c\to \infty$.

\begin{lemma}\label{lem:lim} 
    Let $t\leq 0$ be fixed and $\varepsilon >0$, then
        \[\lim_{c\to \infty}\frac{I_{(-\infty, t]}^{b}(n,m)(c)}{I^{b}_{[t-\varepsilon,t]}(n,m)(c)} = 1.\]
    If additionally $b(t) \neq b'(t)$, then
        \[\lim_{c\to \infty}\frac{D_{(-\infty, t]}(n,m)(c)}{D_{[t-\varepsilon,t]}(n,m)(c)} = 1.\]
\end{lemma}

\begin{proof}
    Let $(c_i)$ be a sequence with $c_i \to \infty$ for $i\to \infty$. We first show that for all $\varepsilon >0$ and $t\leq 0$ it holds
    \begin{equation}\label{eq:i1}
            \lim_{i\to \infty}\frac{I_{(-\infty, t]}^{b}(n,m)(c_i)}{I^{b}_{[t-\varepsilon,t]}(n,m)(c_i)} = 1.
    \end{equation}  
    Heuristically, this means that only the parts with the largest variances (i.e., $\frac{1}{-s}$) governs the limits.
    Since $I^{b}_{(-\infty, t]}(n,m) = I^{b}_{[t-\varepsilon,t]}(n,m) + I^{b}_{(-\infty,t-\varepsilon)}(n,m)$, equation \eqref{eq:i1} is equivalent to
        \[\lim_{i\to \infty}\frac{I^{b}_{(-\infty,t-\varepsilon]}(n,m)(c_i)}{I^{b}_{[t-\varepsilon,t]}(n,m)(c_i)} = 0.\]
    The numerator and denominator are positive. We derive an upper bound for $I^{b}_{(-\infty,t-\varepsilon]}(n,m)(c)$.  Let $h_1 = \max\{\tilde h(y)\mid y\in[0,b_\infty]\}$ then
    \begin{align*}
         I^{b}_{(-\infty,t-\varepsilon]}(n,m)(c) 
         & =\int_{-\infty}^{t-\varepsilon}\int_{0}^{b(s)}(-s+1)^{n}y^{m}e^{cy+\frac{1}{2}c^2s-rs}\tilde h(y)\d y \d s\\
        & \leq \int_{-\infty}^{t-\varepsilon}\int_{0}^{b_{\infty}}(-s+1)^{n}b_\infty^{m}e^{c b_\infty+\frac{1}{2}c^2s-rs}h_1\d y \d s\\ 
        & = N_1 e^{cb_\infty}\int_{-\infty}^{t-\varepsilon}(-s+1)^{n}e^{\frac{1}{2}c^2s-rs}\d s \\
        & = N_1 e^{cb_\infty}\frac{1}{\frac{c^2}{2}-r} P(c,\varepsilon) e^{(\frac{c^2}{2}-r)(t-\varepsilon)}
    \end{align*}
    where $N_1$ is a constant and $P$ is a polynomial in $\frac{1}{c}$ and $\varepsilon$.
    
    We now derive a lower bound for $I^{b}_{[t-\varepsilon,t]}(n,m)(c)$. Let $ 0< a_1 <a_2 \leq b(t-\frac{\varepsilon}{2})$ such that $\tilde h(y)$ is strictly positive on $[a_1,a_2]$ and let $h_2 = \min\{\tilde h(y)\mid y\in[a_1,a_2]\}>0$. Then,
    \begin{align*}
        & \int_{t-\varepsilon}^{t}\int_{0}^{b(s)}(-s+1)^{n}y^{m}e^{cy+\frac{1}{2}c^2s-rs}\tilde h(y)\d y \d s\\
        & \quad \geq \int_{t-\varepsilon}^{t-\frac{\varepsilon}{2}}\int_{a_1}^{a_2}(-t+\varepsilon+1)^{n}a_1^{m}e^{0+\frac{1}{2}c^2s-rs}h_2\d y \d s\\
        & \quad =  N_2 \int_{t-\varepsilon}^{t-\frac{\varepsilon}{2}}e^{\frac{1}{2}c^2s-rs}\d s \\
        & \quad= N_2 \left(e^{(\frac{c^2}{2}-r)(t-\frac{\varepsilon}{2})} - e^{(\frac{c^2}{2}-r)(t-\varepsilon)}\right)
    \end{align*}
    where $N_2>0$ is a constant.
    Putting these results together we have
    \begin{align*}
        0\leq \lim_{i\to \infty}\frac{I^{b}_{(-\infty,-\varepsilon)}(n,m)(c_i)}{I^{b}_{[-\varepsilon,0]}(n,m)(c_i)} 
        \leq \lim_{i\to \infty}\frac{N_1 e^{c_ib_\infty}\frac{1}{\frac{c_i^2}{2}-r} P(c_i,\varepsilon) e^{(\frac{c_i^2}{2}-r)(t-\varepsilon)}}{N_2 \left(e^{(\frac{c_i^2}{2}-r)(t-\frac{\varepsilon}{2})} - e^{(\frac{c_i^2}{2}-r)(t-\varepsilon)}\right)} 
        = 0 
    \end{align*}   
    which shows the first claim. \\
    Let now $b(t) \neq b'(t)$. We assume w.l.o.g.\ that $b(t) > b'(t)$. Since $b$ and $b'$ are continuous, there exists $\delta>0$ such that $b(s) > b'(s)$ for all $s\in [t-\delta,t]$.
    Then, the term
        \[D_{[t-\delta,t]}(n,m)(c_i) = \int_{t-\delta}^{t}\int_{b'(s)}^{b(s)}(-s+1)^{n}y^{m}e^{cy+\frac{1}{2}c^2s-rs}\tilde h(y)\d y \d s\]
    is positive and analogously to the calculations above we obtain 
        \[ \lim_{i\to \infty}\frac{|D_{(-\infty, t-\delta]}(n,m)(c_i)|}{D_{[t-\delta,t]}(n,m)(c_i)} 
        \leq  \lim_{i\to \infty}\frac{I^{\max(b,b')}_{(-\infty, t-\delta]}(n,m)(c_i)}{D_{[t-\delta,t]}(n,m)(c_i)} = 0.\]
    If $\varepsilon \leq \delta$ we can set $\delta = \varepsilon$ and we are done. If $\varepsilon > \delta$, we have
    \begin{align*}
         \lim_{i\to \infty}\frac{|D_{(-\infty, t-\varepsilon]}(n,m)(c_i)|}{D_{[t-\varepsilon,t]}(n,m)(c_i)} 
        \leq \lim_{i\to \infty}\frac{|D_{(-\infty, t-\varepsilon]}(n,m)(c_i)|}{D_{[t-\delta,t]}(n,m)(c_i) - I^{\max(b,b')}_{(-\infty, t-\delta]}(n,m)(c_i)}
         = 0.
    \end{align*}
    For $c$ large enough the denominator is positive, hence, it follows     
        \[\lim_{i\to \infty}\frac{D_{(-\infty, t]}(n,m)(c_i)}{D_{[t-\varepsilon,t]}(n,m)(c_i)} = 1.\]
\end{proof}
The next lemma states that if the integral on the right-hand side of \eqref{eq:main22} is the same for two different stopping boundaries, then this property is not changed if we multiply the integrand with a polynomial $q(s,y)$.
       
\begin{lemma}\label{lem:1}
    If $I^{b}(c) = I^{b'}(c)$ for all $c>\sqrt{2r}$, then
        \[I^{b}(n,m)(c) = I^{b'}(n,m)(c)\]
    for all $c>\sqrt{2r}$ and $n,m\in\nn_{0}$.
\end{lemma}
\begin{proof}
    We prove the claim by induction. Let $n,m=0$ then $I^{b}(n,m)(c) = I^{b'}(n,m)(c)$ by assumption. 
    Let now $n,m\neq 0$ be fixed and $I^{b}(l,k)(c) = I^{b'}(l,k)(c)$, for all $c>\sqrt{2r}$ and $l\leq n, k\leq m$.
    We multiply 
        \[I^{b}(n,m)(c) = \int_{-\infty}^{0}\int_{0}^{b(s)}(-s +1)^{n}y^{m}e^{cy+\frac{1}{2}c^2s-rs}\tilde h(y)\d y \d s\]
    with $e^{-\frac{c^2}{2}}$, take the derivative in $c$ and we have  
    \begin{align*}
        &\frac{\partial }{\partial c} \int_{-\infty}^{0}\int_{0}^{b(s)}(-s+1)^{-n}y^{m}e^{cy+\frac{1}{2}c^2(s-1)-rs}\tilde h(y)\d y \d s  \\
        & \quad =  \int_{-\infty}^{0}\int_{0}^{b(s)}\big(y+ c(s-1)\big)(-s+1)^{-n}y^{m}e^{cy+\frac{1}{2}c^2(s-1)-rs}\tilde h(y)\d y \d s\\
        & \quad =  e^{-\frac{c^2}{2}}\left(  I^{b}(n,m+1)(c) - cI^{b}(n+1,m)(c)\right).
    \end{align*}
    For $b'$ we obtain the same result and together we have
    \begin{equation}\label{eq:c}
        c\big(I^{b}(n+1,m)(c) - I^{b'}(n+1,m)(c)\big) = I^{b}(n,m+1)(c) - I^{b'}(n,m+1)(c).
    \end{equation}
    Let us assume that $I^{b}(n+1,m) \neq I^{b'}(n+1,m)$.
    $I^{b}(n+1,m)$ is analytic in $c$, hence there exists a series $(c_i)$ with $c_i\to \infty$ such that
    \[I^{b}(n+1,m)(c_i) \neq I^{b'}(n+1,m)(c_i)\]
    for all $i\in\nn$. We can solve \eqref{eq:c} for $c_i$ and find
    \begin{equation}\label{eq:ci}
        c_i = \frac{I^{b}(n,m+1)(c_i) - I^{b'}(n,m+1)(c_i)}{I^{b}(n+1,m)(c_i) - I^{b'}(n+1,m)(c_i)}
        = \frac{D(n,m+1)(c_i)}{ D(n+1,m)(c_i)}.
    \end{equation}
    We show that the right-hand side is bounded, in contradiction to the assumption $c_i\to\infty$.\\
    Let $t^{\ast} = \sup\{t\mid b(t) \neq b'(t)\}$. If the set is empty, there is nothing to show. Let $(t_j)_{j\in\nn}$ be a non-decreasing sequence with $b(t_j) \neq b'(t_j)$ and $\lim_{j\to \infty} t_j = t^{\ast}$.
    
    \begin{figure}[htbp]
        \centering
        \includegraphics[width=0.7\textwidth]{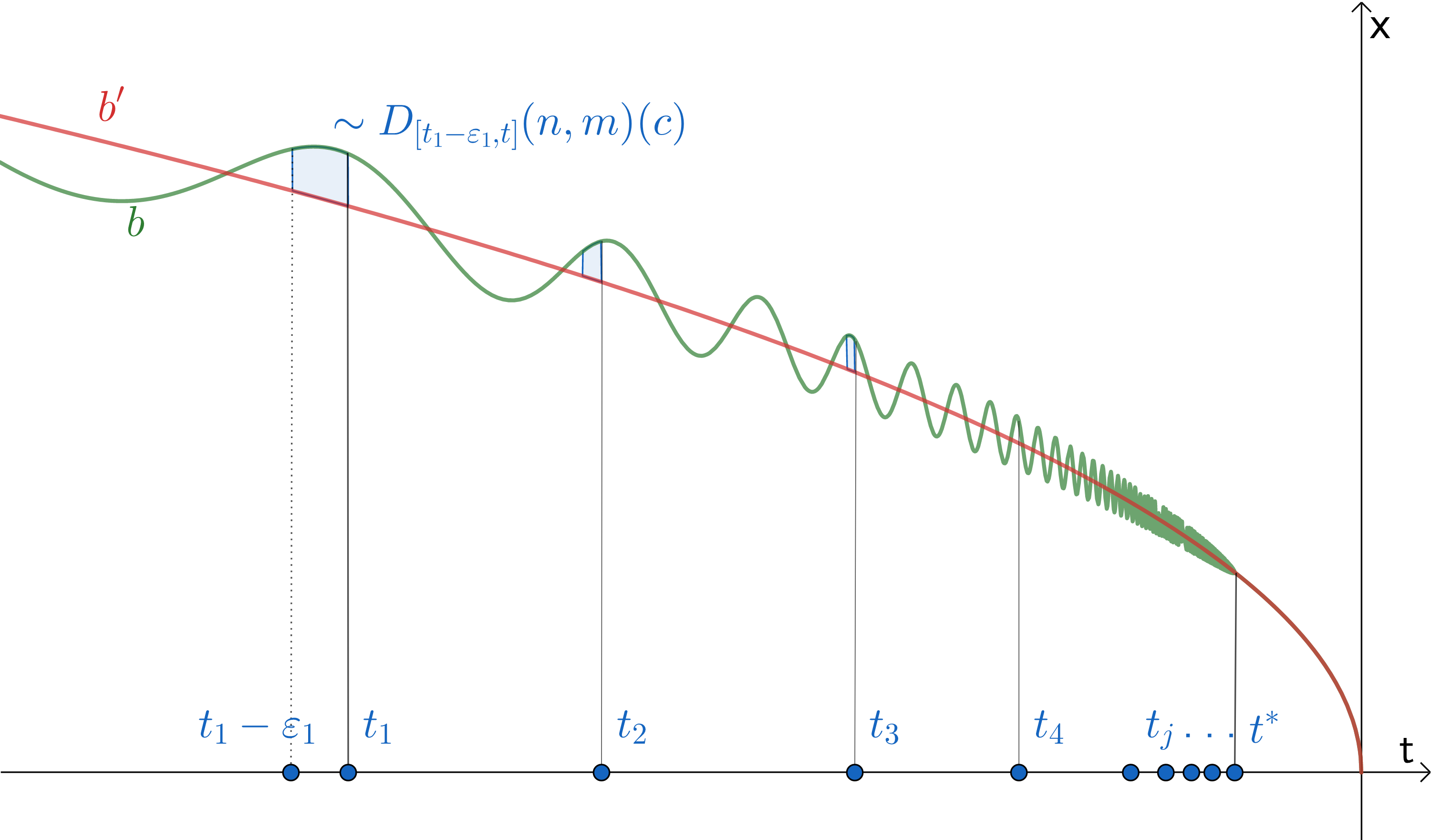}
        \caption{Visualization of the setting in Lemma \ref{lem:1}}
        \label{fig:tj}
    \end{figure}

   For every $t_j$ we find $t_{j+1}-t_j>\varepsilon_j>0$ such that $b(s) \neq b'(s)$ for all $s\in [t_j-\varepsilon_j,t_j]$. We use Lemma \ref{lem:lim} and the fact that $b$ is decreasing to obtain
    \begin{align*}
        \lim_{i\to\infty}\frac{D_{(-\infty,t_j]}(n,m+1)(c_i)}{ D_{(-\infty,t_j]}(n+1,m)(c_i)}
       & = \lim_{i\to\infty}\frac{D_{[t_j-\varepsilon_j,t_j]}(n,m+1)(c_i)}{ D_{[t_j-\varepsilon_j,t_j]}(n+1,m)(c_i)} \\
       &   \leq \lim_{i\to\infty}\frac{b(t_j-\varepsilon_j) D_{[t_j-\varepsilon_j,t_j]}(n,m)(c_i)}{ (t_j + 1) D_{[t_j-\varepsilon_j,t_j]}(n,m)(c_i)}\\    
       & = \lim_{i\to\infty}\frac{b(t_j-\varepsilon_j) }{ (t_j + 1)}\\
       & \leq b(t_j-\varepsilon_j).
    \end{align*}
    The sequence $(b(t_j-\varepsilon_j))_{j\in\nn}$ is decreasing and since $b$ is continuous we have
    \[\lim_{j\to\infty}b(t_j-\varepsilon_j) = b(t^{\ast})<\infty.\]
    By dominated convergence we have
      \[\lim_{j\to \infty}\lim_{i\to \infty}\frac{D_{(-\infty,t_j]}(n,m+1)(c_i)}{ D_{(-\infty,t_j]}(n+1,m)(c_i)}  
        =  \lim_{i\to \infty}\frac{D_{(-\infty,t^\ast]}(n,m+1)(c_i)}{ D_{(-\infty,t^{\ast}]}(n+1,m)(c_i)} 
        <  \infty.\]
    For all $t > t^{\ast}$ we have $b(t) = b'(t)$, hence
        \[\lim_{i\to\infty} \frac{D(n,m+1)(c_i)}{ D(n+1,m)(c_i)} <\infty,\]
    in contradiction to \eqref{eq:ci}. It follows 
    $I^{b}(n+1,m) = I^{b'}(n+1,m)$. Plugging this into \eqref{eq:c} we see that also $I^{b}(n,m+1) = I^{b'}(n,m+1)$.
    
\end{proof}

\noindent We can now state our main theorem on the uniqueness of $b$.

\begin{theorem}\label{th:eind}
    Given a stopping problem that fulfills Assumptions \ref{as:1}. If there is $N\in\rr$ such that the Fredholm representation
    \eqref{eq:main} holds for all $c > N$, then $b$ is determined  uniquely by \eqref{eq:main} in the class of continuous monotone functions.
\end{theorem}

\begin{proof}
    Let $I^{b}(c) = I^{b'}(c)$ for all $c>\sqrt{2r}$. By Lemma \ref{lem:1} we have $I^{b}(n,m)(c) = I^{b'}(n,m)(c)$ for all $c>\sqrt{2r}$ and $n,m\in \rr$. We rewrite $I^{b}(n,m)(c)$ as
        \[I^{b}(n,m)(c) = \int e^{s}(-s +1)^{n}y^{m} \mu(ds,dy)
        \]
    with a measure $\mu$ with
        $\frac{\partial \mu}{\partial \lambda} =f$
    where $f$ denotes the density function
        \[f(s,y) := \ind_{\{0\leq y\leq b(t), t\leq 0\}} e^{-s}e^{cy+\frac{1}{2}c^2s-rs}\tilde h(y)\]
    and $\lambda$  denotes the Lebesgue measure on $\rr^{2}$. A measure $\mu'$ with density function $f'$ is defined analogously via $b'$.\\
    Functions of the form $p(-s+1)e^{s}$ for a polynomial $p$ are called \textit{Laguerre exponential polynomials} and they lie dense in $L^{2}(\rr_{\leq 0})$, see \cite[Lemma 1. (ii)]{ang72}. For  $B:=\rr_{\leq0}\times [0,M]$, the family $\{q(-s,y)e^{s}|q \text{ is a polynomial on } B\}$ lies dense in $L^{2}(B)$.
    By Lemma \ref{lem:1} and linearity of the integral we have
        \[\int e^{s}q(-s,y) f(s,y) \d(s,y) = \int e^{s}q(-s,y) f'(s,y) \d(s,y)\]
    for all polynomials $q$. Let now $g \in L^{2}(B)$ and $q_n$ be a sequence of polynomials with $e^{s}q_n \overset{L^2} \to g$. For $c$ large enough, $f$ and $f'$ are positive, bounded and in $L^{2}$. It follows that
        \[\lim_{n\to\infty}\int e^{s}q_n(-s,y) f(s,y) \d (s,y)
        = \int g(s,y) f(s,y) \d(s,y) \]
    and
        \[\lim_{n\to\infty}\int e^{s}q_n(-s,y) f'(s,y) \d(s,y)
        =\int g(s,y) f'(s,y) \d(s,y),\]
    hence,
        \[\int g(s,y) \mu(ds,dy) = \int g(s,y) \mu'(ds,dy)\]
    for all $g \in L^{2}(B)$.\\
    If follows that $\mu = \mu'$ -a.e. Since $b$ and $b'$ are continuous and $\lambda(\{\tilde h(y) = 0 \}) = 0 $ we have $b = b'$.
\end{proof}

\begin{remark}
    It is not straightforward to extend the proof immediately to the general setting \eqref{eq:main}, but some assumptions are easy to relax.
    \begin{itemize}
        \item We used that $b$ is bounded in the discounted case. For the proof it is enough if $b$ grows at most linearly for $t\to -\infty$. That would then include cases like Example \ref{ex:stadje}.
        \item We used monotonicity of $b$ for Lemma \ref{lem:1}. It would be enough to assume that $b$ is bounded on compact intervals, e.g., $b$ continuity would be sufficient. 
        \item The assumption $\lambda(\{\tilde h(y) = 0 \}) = 0 $ is not necessary. If we do not assume this we would get uniqueness up to sets $H$ with $\lambda(\{y\in H\mid \tilde h(y)\neq 0\})=0$.
        \item The assumption that \eqref{eq:main} holds for all $c>N$ can be relaxed to \eqref{eq:main} holds for $(c_i)=(c_i)_{i\in\nn}$ where the set of limit points of $(c_i)$ is not bounded.
    \end{itemize}
\end{remark}


\section{Numerics}\label{sec:5}
We show how the Fredholm representation can be used to numerically  approximate the optimal stopping boundary $\partial C$. We stick to the one-dimensional and one-sided case where we approximate the function $d$.
We discretize \eqref{eq:main} -- in $y$ and in $c$ --
and use some optimization algorithm to minimize
\begin{equation}\label{eq:num1}
   f(d) = \sum_{l=1}^{M} F_{c_l}\left(\int_{-\infty}^{0} e^{c_l y}\tilde h(y)\d y +  \sum_{n=1}^{N-1}e^{(\frac{c^2}{2}-r)d_n} \int_{y_n}^{y_{n+1}} e^{y}\tilde h(y)\d y\right)
\end{equation}
numerically.
Here $M$ and $N$ are constants and $F_{c_l}$ are convex functions.
%
%

\paragraph{Upper and lower bounds}
For the numeric minimization to be stable it is very helpful to have suitable upper and lower bounds for $d$. Clearly, $d^{u}_{1}:=0$ is an upper bound, hence 
    \[\int_{-\infty}^{b_{\infty}}e^{(\frac{c^2}{2}-r)d^{u}_{1}(y)}  e^{cy}\tilde h(y)\d y \geq 0\]
for all $c>\sqrt{2r}$. We define a lower bound for $x>0$ via
    \[d^{l}_{1}(x): = \inf\left\{t\mid \int_{-\infty}^{b_{\infty}}e^{(\frac{c^2}{2}-r)\tilde d^{u}_{1}(y,x)}  e^{cy}\tilde h(y)\d y \geq 0,~\forall c>\sqrt{2r}\right\}\]
where 
    $\tilde d^{u}_{1}(y,x): = t 1_{[x,b_{\infty}]}(y)\wedge d^{u}_{1}(y).$
Since $d^{l}_{1}$ is a lower bound we have
    \[\int_{-\infty}^{b_{\infty}}e^{(\frac{c^2}{2}-r)d^{l}_{1}(y)}  e^{cy}\tilde h(y)\d y \leq 0\]
for all $c>\sqrt{2r}$. We can now find a better upper bound via
    \[d^{u}_{2}(x): = \sup\left\{t\mid \int_{-\infty}^{b_{\infty}}e^{(\frac{c^2}{2}-r)\tilde d^{l}_{1}(y,x)}  e^{cy}\tilde h(y)\d y \leq 0,~\forall c>\sqrt{2r}\right\}\]
where 
    $\tilde d^{l}_{1}(y,x): = t 1_{[0,x]}(y)\vee d^{l}_{1}(y)$.
Repeating this procedure improves the bounds slightly. They, however, do not converge to $d$.

\begin{example}\label{ex:num}
    Let $h(x) = x$ and $r=1$. We have $\tilde h(x) = x$,
    $b_0 = 0$, $b_\infty = \sqrt{\frac{1}{2}}$ and
        $\int_{-\infty}^{0} e^{cy}\tilde h(y)\d y = \frac{1}{c^2}$.
    We set 
        $F_c(x) = \left(c^2 x +\frac{1}{1+c^2 x}\right)^{2},$
    $N=60$, $M=40$, $y_k =  \frac{b_\infty(k-1)}{N-1}$, and $c_l = \sqrt{2r} + \frac{l}{10}$. A plot of the resulting stopping boundary is given in Figure \ref{fig:d2}. The plot nicely reflects the boundary behavior described in Section \ref{sec:3}. The function $-B_1x^2 \approx -2.4503 x^2$ is plotted as a reference.
    \begin{figure}[htbp]
        \centering
        \includegraphics[width=0.7\textwidth]{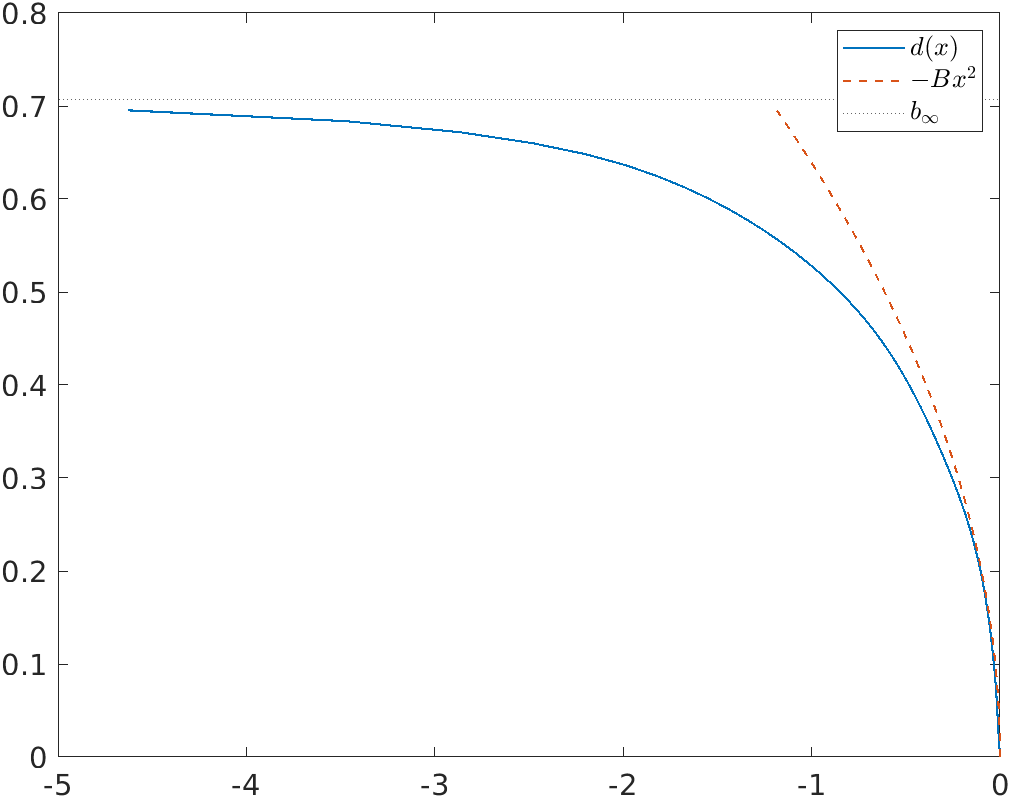}
        \caption{The stopping boundary from Example \ref{ex:num}. As a reference the \emph{limit} function $-B_1x^2$ and $b_{\infty}$ are plotted as well.}
        \label{fig:d2}
    \end{figure}
\end{example}

\section{Conclusion}\label{sec:con}
We believe that the Fredholm representation presented here is a useful tool for analyzing optimal stopping problems with finite time horizon. We have seen that in some cases the continuation set of a stopping problem is fully characterized by the representation. We are confident that uniqueness holds in a more general setting as well, but this still has to be proven. 

A comparison of the integral equation presented here with the otherwise mostly used one is not quite easy. The latter has proved to be extremely fruitful and universally applicable in the last decades. We see the equation presented here rather as a supplement. In the analysis of new problems, it can be used in addition to the Volterra equation. 

Section \ref{sec:3} illustrates that very general asymptotic results can be obtained on the basis of the Fredholm integral equation. A comparison with existing asymptotic results seems to show that our analysis is more straightforward and at the same time can be made mathematically precise with manageable effort. The Fredholm representation also seems to be useful for numeric evaluation of stopping problems. One advantage over the standard integral representation is that the kernel has no singularities what makes numeric integration more stable.

\appendix
\section{A proof}\label{sec:appendix}
\begin{proof}[Proof of interchangeability of limit and integral in the proof of Theorem \ref{th:martin}]~\\
    We want to use the dominated convergence theorem. 
    To do so we split the continuation set $C$ into two parts. On one part
    the integrand in \eqref{eq:lim1} is dominated by the limit function $e^{\textbf{c}\cdot \textbf{y} + \frac{\norm{\textbf{c}}^2}{2}s}e^{-rs} \tilde h (\textbf{y})$. The other part we denote by $A_t$ and we construct an integrable upper bound in the following.
%
    Let $\textbf{c} \in \rr^{n}$ such that $\{(t,-t\textbf{c})\mid t\leq 0\}\cap C$ is bounded and the integral
        \[0 = \int_{-\infty}^{0}\int_{C_s}e^{\textbf{c}\cdot \textbf{y} + \frac{\norm{\textbf{c}}^2}{2}s}e^{-rs} \tilde h (\textbf{y})\d \textbf{y} \d s\]
    exists.
    We define 
    \begin{equation}\label{eq:majconv}
        A_t: = \left\{(s,\textbf{y})\mid t<s<0, \, \left(\frac{-t}{s-t}\right)^{\frac{n}{2}} e^{\frac{2\textbf{c}\cdot \textbf{y} t^2 + t\norm{\textbf{y}}^2 +\norm{\textbf{c}}^2t^2s }{2(t^2-st)} }
        > e^{\textbf{c}\cdot \textbf{y} + \frac{\norm{\textbf{c}}^2}{2}s}\right\}.
    \end{equation}
    We rearrange the defining inequality in \eqref{eq:majconv} to
    \begin{align*}
        \left(\frac{-t}{s-t}\right)^{\frac{n}{2}}e^{-\frac{\norm{\textbf{c}s+\textbf{y} }^2}{2(s-t)} }& >1\mbox{, i.e.,} \quad \norm{\textbf{c}s+\textbf{y} }
        <\sqrt{n (t-s)\ln\left(1-\frac{s}{t}\right)},
    \end{align*}
    which converges to
        $\norm{\textbf{c}s+\textbf{y}} <\sqrt{n(-s)}$
    as $t\to -\infty$. Since the logarithm is concave, we have $A_t\subset A_u$ for $u<t$, and hence
        \[A_{\infty}:=\bigcup_{t<0}A_t = \left\{ (s,\textbf{y})\mid \norm{\textbf{c}s+\textbf{y} }
    <\sqrt{n(-s)}\right\}.\]
    Let $C_{\infty}$ be the continuation set of the infinite time horizon problem. We have $C_t\subset C_{\infty}$ for all $t\leq 0$. Now  $\{(t,-t\textbf{c})\mid t\leq 0\}\cap C$ bounded implies that $A_\infty\cap C$ is bounded.
    For $t$ small enough the functions 
        \[(s,\textbf{y})\mapsto \left(\frac{-t}{s-t}\right)^{\frac{n}{2}}e^{-\frac{\norm{\textbf{c}s+\textbf{y} }^2}{2(s-t)} }\]
    are all bounded on $A_{\infty}$ by some $M<\infty$.
    So, on $C$ we have that
        \[M\left| e^{\textbf{c}\cdot \textbf{y} + \frac{\norm{\textbf{c}}^2}{2}s}(-\g g)(s,\textbf{y}) \right|\]
    is an integrable upper bound for 
    \[\left|\left(\frac{-t}{s-t}\right)^{\frac{n}{2}} e^{\frac{2\textbf{c}\cdot \textbf{y} t^2 + t\norm{\textbf{y}}^2 +\norm{\textbf{c}}^2t^2s }{2(t^2-st)} }(-\g g)(s,\textbf{y}) \right|.\] 
    The result follows by the dominated convergence theorem.
\end{proof}

\bibliographystyle{alpha}
\bibliography{a}

\end{document}